\documentclass[preprint]{elsarticle}

\usepackage{amsmath,amssymb,amsfonts,amsthm,amstext,amsbsy,amsopn,amscd}
\usepackage{newlfont}

 \newcommand{\dist}{\mathrm{d} }
 \newcommand\calP{{\mathcal P}}
 \newtheorem{thm}{Theorem}[section]
 \newtheorem{cor}[thm]{Corollary}
 \newtheorem{lem}[thm]{Lemma}
 \newtheorem{prop}[thm]{Proposition}
 \newtheorem{prob}[thm]{Problem}

 \theoremstyle{definition}
 
 \newtheorem{hyp}[]{Hypothesis}
 \newtheorem{rem}[thm]{Remark}

 \newcommand{\diam}{\mathrm{diam}}
 \newcommand{\girth}{\mathrm{girth} }
 \newcommand{\Ga}{\Gamma}
 \newcommand{\Si}{\Sigma}

 \newcommand{\Aut}{\mathrm{Aut}\,}

 \newcommand{\PGaL}{\mathrm{P\Gamma L}}

 \newcommand{\PGL}{\mathrm{PGL}}
 
 \newcommand{\PSU}{\mathrm{PSU}}

 \newcommand{\HoSi}{\mathrm{HoSi}}
 \newcommand{\rmP}{\mathrm{P}}
 \def\S{\textsf{S}}

 \def\ssig{\S(\Sigma)}

\journal{Discrete Mathematics}

\begin{document}

\begin{frontmatter}



\title{Symmetry properties of subdivision graphs}

 \author[a1]{Ashraf Daneshkhah\corref{cor1}}
 \ead{adanesh@basu.ac.ir}
 \author[a2]{Alice Devillers} \ead{alice.devillers@uwa.edu.au}
 \author[a2]{Cheryl E. Praeger} \ead{cheryl.praeger@uwa.edu.au}
 \address[a1]{Department of Mathematics, Faculty of Science, Bu-Ali Sina University, Iran.}
 \address[a2]{Centre for the Mathematics of Symmetry and Computation, School of Mathematics and Statistics, University of Western Australia, 
 Perth, Australia.}
 \cortext[cor1]{Corresponding author}
 


\begin{abstract}
The subdivision graph  $\S(\Sigma)$ of a graph  $\Sigma$ is obtained from $\Sigma$ by `adding a vertex' in the middle of every edge of $\Si$. Various symmetry properties of $\S(\Sigma)$ are studied. We prove that, for a connected graph $\Sigma$, $\S(\Sigma)$ is locally $s$-arc transitive if and only if $\Sigma$ is $\lceil\frac{s+1}{2}\rceil$-arc transitive. The diameter of $\S(\Sigma)$ is $2d+\delta$, where $\Sigma$ has diameter $d$ and $0\leqslant \delta\leqslant 2$, and local $s$-distance transitivity of $\S(\Sigma)$ is defined for $1\leqslant s\leqslant 2d+\delta$.  In the general case where $s\leqslant 2d-1$ we prove that
$\S(\Sigma)$ is locally $s$-distance transitive if and only if $\Sigma$ is $\lceil\frac{s+1}{2}\rceil$-arc transitive. For the remaining values of $s$, namely $2d\leqslant s\leqslant 2d+\delta$, we classify the graphs $\Sigma$ for which $\S(\Sigma)$ is locally $s$-distance transitive in the cases, $s\leqslant 5$ and $s\geqslant 15+\delta$. The cases $\max\{2d, 6\}\leqslant s\leqslant \min\{2d+\delta, 14+\delta\}$ remain open.
\end{abstract}


\end{frontmatter}



\section{Introduction}

All graphs in this paper are simple and undirected, that is to say, a graph $\Sigma$ consists of a  vertex-set $V(\Sigma)$ and a subset  $E(\Sigma)$ of unordered pairs from $V(\Sigma)$, called edges.
The {\it subdivision graph} $\S(\Si)$ of a graph $\Si$  is defined as the graph with vertex set $V(\Sigma) \cup E(\Sigma)$ and edge set $
\{\{x, e\}| \ x\in V(\Sigma), \  e\in E(\Sigma), \ x\in e\}.
$
Informally, $\ssig$ is the graph obtained from $\Sigma$ by `adding a vertex' in the middle of each edge of $\Si$. The purpose of this paper is to elucidate connections between various symmetry properties of $\Sigma$ and of its subdivision graph $\S(\Sigma)$, in particular local $s$-arc-transitivity, and local $s$-distance transitivity (defined in Section \ref{sec-pre}).
These properties are generalisations and natural analogues of graph properties extensively studied in the literature, namely $s$-arc transitivity introduced by W.T. Tutte~\cite{Tutte47,Tutte59} and
distance transitivity introduced by D.G. Higman \cite{Higman67}, and  studied further by N. Biggs and  many others (see \cite{Biggs74,Cameron83,Praeger87,van07, Weiss1981}). Moreover the families of locally $s$-arc transitive graphs and locally $s$-distance transitive graphs were analysed in \cite{GLP} and \cite{DGLP} respectively.

In Section~\ref{sec-pre} we give basic graph theoretic concepts and notation, including definitions of the transitivity properties we study.
Some basic properties of $\ssig$ are given in Section~\ref{sec-local}. Firstly $\ssig$ is  bipartite and is connected if $\Si$ is connected, and the graph $\Sigma$ can be reconstructed from $\ssig$ (in the exceptional case where $\Sigma$ is a cycle, reconstruction of $\Si$ from $\ssig$ is up to isomorphism only). Cycles arise as exceptions in other ways also. The automorphism group  $\Aut (\Si)$ acts on $V(\Si)$ and $E(\Si)$ and preserves the incidence of $\S(\Si)$ giving a natural embedding  $\Aut(\Si)\leqslant \Aut (\S(\Si))$. With the exception of cycles, this embedding is an isomorphism. For cycles $\Si=C_{n}$, $\ssig=C_{2n}$ and $\Aut(\ssig)=D_{4n}=(\Aut(\Sigma)).2$.
Note that the subdivision graphs of cycles are vertex transitive, while
for all other graphs $\Si$, $\Aut(\S(\Si))=\Aut(\Sigma)$ fixes $V(\Si)$ and $E(\Si)$ setwise.

 In Section ~\ref{sec:proofs1}, we prove a decisive relationship between the levels of arc-transitivity of $\Si$ and $\ssig$.

\begin{thm}\label{thm:bary}
Let $\Sigma$ be a connected graph, $s$ a positive integer, and $G\leqslant\Aut(\Sigma)$.
Then $\S(\Si)$ is locally $(G, s)$-arc  transitive if and only if
$\Sigma$ is $(G, \lceil\frac{s+1}{2}\rceil)$-arc transitive.
\end{thm}

This generalises Theorem 3.10 of \cite{GLP}, which proves the result for odd $s$. The concepts of $(G,1)$-arc transitivity and $(G,1)$-distance transitivity are equivalent, as are their local variants. Thus
Theorem \ref{thm:bary} implies the equivalence of  $(G,1)$-arc transitivity of $\Si$ and local $(G,1)$-distance transitivity of $\ssig$.
Moreover, if $s=2$ or $3$, then the local $(G,s)$-distance transitivity of $\ssig$ is equivalent to the $(G,2)$-arc transitivity of $\Si$, see Corollary \ref{basic}.
A similar relationship holds for larger values of $s$ provided  $s\leqslant 2\,\diam(\Si)-1$ (proved in Section~\ref{sec:proofs1}).

\begin{thm}\label{main3}
Let $\Sigma$ be a connected graph and $s$ a positive integer such that $s\leqslant 2\,\diam(\Si)-1$ (so in particular  $\diam(\S(\Si))> s$). Then $\S(\Si)$ is locally $(G,s)$-distance transitive  if and only if $\Sigma$ is
$(G, \lceil \frac{s+1}{2}\rceil)$-arc transitive.
\end{thm}

We denote the diameter of $\Si$ by $d$. The diameter of $\ssig$ is $2d+\delta$ with $\delta\in\{0,1,2\}$ (see Remark~\ref{prop-main1}).
Thus the $s$-values for which Theorem~\ref{main3} gives no information are those satisfying $2d\leqslant s\leqslant 2d+\delta$. For very large and very small values of $s$ in this range we can determine explicitly the pairs $\Si, G$ for which $\ssig$ is locally $(G,s)$-distance transitive. In this result, $\rmP$ and $\HoSi$ denote the Petersen graph and the Hoffman-Singleton graph, respectively.

\begin{thm}\label{thm-char}
Let $\Si$ be a connected graph with $|V(\Si)|\geqslant2$ and diameter $d$, let $G\leqslant \Aut(\Sigma)$, and let $s$ be a positive integer such that $2d\leqslant s\leqslant 2d+\delta=\diam(\ssig)$. 
\begin{enumerate}
\item[(a)] If $s\geqslant 15+\delta$, then $\ssig$ is  locally $(G,s)$-distance transitive if and only if $\Si=C_n$ and $G=D_{2n}$, either with $n=s$, or with $n=s+1$ odd.
\item[(b)] If $s\leqslant 5$, then $\ssig$ is  locally $(G,s)$-distance transitive if and only if $\Si$ and $G$ are as in Table {\rm\ref{table}}.
\end{enumerate}
\end{thm}
\begin{table}[]\label{table}
\caption{$\Si$ and $G$ for $s=2,3,4,5$}\smallskip
\begin{tabular}{|l|l|l|l|l|l|}
\hline
   $\Si$ &  $s$ & $d$ & $\delta$ & $\Aut(\Si)$ &$G$  \\
\hline
$K_2$ & $2$   & 1 & 0 & $S_2$ & $S_2$ \\
$K_3$ & $2,3$ & 1 & 1 & $S_3$ & $S_3$ \\
$K_n$, ($n\geqslant4$) &$2,3$ &$1$& $2$   &$S_{n}$ & $3$-transitive subgroup of $S_n$   \\
          &$4$ &  &      &   & $4$-transitive subgroup of $S_n$  or\\
	  &    &  &      &   &  $n=9$ and $G=\PGaL(2,8)$\\
$K_{n,n}$, ($n\geqslant2$) &$4$ &$2$& $0$ & $S_{n}\wr S_{2}$& Conditions of Example~\ref{ex-Knn}         \\
$C_{5}$&  $4,5$ & 2 & 1   &$D_{10}$ & $D_{10}$   \\
$\rmP$    &  $4,5$ & 2 & 2   &$S_{5}$& $S_{5}$    \\
$\HoSi$ &  $4,5$ & 2 & 2   &$\PSU(3,5).2$& $\PSU(3,5)$ or $\PSU(3,5).2$      \\
\hline
\end{tabular}
\end{table}

The first part of Theorem~\ref{thm-char} is an application of the deep theorem of R. Weiss \cite{Weiss1981} that the only 8-arc transitive graphs are the cycles, while the second part uses the classification by A.A. Ivanov \cite{Ivanov87} of $3$-arc transitive graphs of girth 5. As R. Weiss's result relies on the finite simple group classification so also does Theorem~\ref{thm-char}(a). The local distance transitivity properties claimed for the graphs $\ssig$ in Table~\ref{table} are established in Section~\ref{sec-examples}.

We prove in fact that, for each of the graphs $\Si$ in Table~\ref{table}, $\ssig$ is locally distance transitive. This enables us to classify all graphs $\Si$ of diameter at most 2 for which $\ssig$ is locally distance
transitive.

\begin{cor}\label{ldt}
Let $\Si$ be a connected graph with $|V(\Si)|\geqslant2$ and $\diam(\Si)\leqslant 2$, and let $G\leqslant \Aut(\Sigma)$. Then $\ssig$ is locally $G$-distance transitive if and only if $\Si,G$ are as in Table~\ref{table} (for the maximum value of $s$).
\end{cor}

The first two authors are considering the unresolved cases of Theorem~ \ref{thm-char}, attacking in particular the problem:

\begin{prob}
Classify graphs $\Si$ for which $\ssig$ is locally distance transitive.
\end{prob}




\section{Concepts and symmetry for graphs}\label{sec-pre}

For a positive integer $s$, an {\it $s$-arc} of a graph $\Sigma$ is an
$(s+1)$-tuple $(v_0,v_1,\ldots, v_s)$ of vertices such that
$\{v_{i-1},v_{i}\}\in E(\Sigma)$ for $1\leqslant i\leqslant s$, and $v_{j-1}\neq
v_{j+1}$ for $1\leqslant j\leqslant s-1$. The integer $s$ is called
the length of the $s$-arc. A $1$-arc is often called an arc.

The {\it distance} between two vertices $v_1$ and $v_2$, denoted by
$\dist_\Sigma(v_1,v_2)$, is the minimum number $s$ such that there
exists an $s$-arc from $v_1$ to $v_2$.
For a connected graph $\Sigma$, 
the {\it diameter of $\Sigma$}, denoted $\diam(\Sigma)$, is the maximum distance between two vertices of $\Si$.
For $0\leqslant i\leqslant\diam(\Sigma)$ define
\[
\Sigma_i:=\{(v,w)\in
V(\Sigma)\times V(\Sigma)\,|\,\dist_\Sigma(v,w)=i\}
\]
and, for $v\in V(\Sigma)$, define
$\Sigma_i(v):=\{w\in V(\Sigma)|\dist_\Sigma(v,w)=i\}$. We often write $\Sigma(v)$ for $\Sigma_1(v)$.
A graph $\Sigma$ is \emph{bipartite} if its vertex set can be partitioned into two non-empty sets called {\it biparts} such that every edge has one vertex in each bipart.

We denote a complete graph and a cycle on $n$ vertices by $K_n$ and $C_n$ respectively;
$K_{m,n}$ denotes the complete bipartite graph with biparts of sizes $m$ and $n$.
A graph is \emph{regular} if its vertices have a constant valency, that is, lie in a constant number of edges. If $\Sigma$ contains a cycle then the \emph{girth} of $\Sigma$ is the length of its shortest cycle.

Let $G\leqslant\Aut(\Sigma)$. The properties we study are defined
relative  to the $G$-action on $\Sigma$. Each property $\calP$ is defined by the requirement that $G$ be transitive on each set in some collection
$\calP(\Sigma)$ of sets, as in Table~\ref{def:trans}. Moreover, each property has a `local variant' that is defined by requiring that, for each
$v\in V(\Sigma)$, the stabiliser $G_v$ be transitive on each set
in a related collection $\calP(\Sigma,v)$ of sets, as in Table~\ref{def:localtrans}. Notice that our definition of local $(G,s)$-arc transitivity is slightly stronger that the definition in \cite{GLP} (actually it is equivalent to our definition as long as $\Si$ has no vertex with valency $1$).


\begin{table}[h]
\begin{center}
\caption{Properties $\calP$ for $G$-action on a connected graph $\Sigma$}\label{def:trans}\smallskip
\begin{tabular}{|l|l|}
\hline
Property $\calP$  & $\calP(\Sigma)=\{\Delta_i|1\leqslant i\leqslant s\}$, and
$\Delta_s\neq\emptyset$\\
\hline
$G$-arc transitivity& $s=1$ and $\Delta_1=\Sigma_{1}$\\
$(G,s)$-arc transitivity& $\Delta_i$ is the set of $i$-arcs of $\Sigma$\\
$(G,s)$-distance transitivity& $\Delta_i=\Sigma_i$\\
$G$-distance transitivity& $s=\diam(\Sigma)$ and $\Delta_i=\Sigma_i$\\
\hline
\end{tabular}
\end{center}
\end{table}

\begin{table}[h]
\begin{center}
\caption{Local properties $\calP$ for $G$-action on a connected graph $\Sigma$} \label{def:localtrans}\smallskip
\begin{tabular}{|l|l|}
\hline
Local property $\calP$  & $\calP(\Sigma,v)=\{\Delta_i(v)|1\leqslant i\leqslant s\}$, and \\
&$\Delta_s(v)\neq\emptyset$ for some $v$\\
\hline
local $G$-arc transitivity& $s=1$ and $\Delta_1(v)=\Sigma_1(v)$\\
local $(G,s)$-arc transitivity& $\Delta_i(v)$ is the set of $i$-arcs of
$\Sigma$ with\\
& initial vertex $v$\\
local $(G,s)$-distance transitivity& $\Delta_i(v)\Sigma_i(v)$\\
local $G$-distance transitivity& $s=\diam(\Sigma)$ and
$\Delta_i(v)=\Sigma_i(v)$\\
\hline
\end{tabular}
\end{center}
\end{table}

These concepts are sometimes used without reference to a particular
group $G$, especially when $G=\Aut(\Sigma)$.
It follows from the definitions that, for $s\geqslant 2$, (local) $(G, s)$-arc transitivity  implies  (local) $(G, s-1)$-arc transitivity, and (local)
$(G,s)$-distance transitivity implies (local) $(G, s-1)$-distance transitivity.

\begin{rem}\label{s=1}
For $s=1$, several of the concepts coincide. Indeed, for a connected graph $\Sigma$ and $G\leq \Aut(\Sigma)$, the properties of local $G$-arc transitivity, local $(G,1)$-arc transitivity and local $(G,1)$-distance transitivity are equivalent.
\end{rem}

\begin{lem}\label{lem-smalls}
Let $\Sigma$ be a connected graph 
and let $G\leqslant\Aut(\Sigma)$.
\begin{enumerate}
\item[(a)] If $G$ is intransitive on $V(\Sigma)$, then each of the equivalent conditions of Remark~\ref{s=1} holds if and only if $G$ is transitive on $E(\Sigma)$. Moreover, in this case,
$\Sigma$ is a bipartite graph and the $G$-orbits in $V(\Sigma)$ are the two biparts.
\item[(b)] $\Sigma$ is locally $(G,2)$-arc transitive if and only if, for all $v\in V(\Sigma)$,
$G_v$ is $2$-transitive on $\Sigma_1(v)$.
\end{enumerate}
\end{lem}
This Lemma is very easy to prove. See for instance \cite[Lemma 3.1, Lemma 3.2]{GLP} for details, noticing that with our definition of local $(G,s)$-arc transitivity, the hypothesis that no vertex has valency $1$ can be dropped.

%

Each of the `global' properties $\calP$ for a graph $\Sigma$ implies the local variant, but the converse is not true in general. For example if $\Sigma=K_{m,n}$ and $G=S_m\times S_n$,
then all the local properties, but none of the global properties, hold with $s=2$.  
However, for each of the local properties, the corresponding global property holds if and only if $G$ is transitive on $V(\Sigma)$ (and in this case, in particular, all vertices have the same valency).



\section{Basic properties of subdivision graphs}\label{sec-local}

We give some further definitions and results related  to graphs that will be used in the next sections. The \emph{line graph} $L(\Sigma)$ of a graph $\Sigma$ is defined as the graph with vertex set $E(\Sigma)$ and edges $\{e_1, e_2\}$, for $e_1, e_2\in E(\Sigma)$ such that $e_1\cap e_2\neq \emptyset$. The \emph{distance $2$ graph} of $\Sigma$ is the graph $\Sigma^{[2]}$ with the same vertex set as $\Sigma$ but with the edge set replaced by the set of all vertex pairs $\{u,v\}$ such that $d_\Sigma(u,v)=2$.  If $\Sigma$ is connected, then all vertices at even $\Sigma$-distance from $v$ lie in the same connected component of $\Sigma^{[2]}$, as do all vertices at odd $\Sigma$-distance from $v$, and so $\Sigma^{[2]}$  has at most two connected components. Moreover, if $\Sigma$ is connected and bipartite, then $\Sigma^{[2]}$  has exactly two components.



Clearly, if $\Sigma$ is connected, then the subdivision graph $\S(\Si)$ is connected and bipartite with biparts $V(\Si)$ and $E(\Si)$. Note that, in $\S(\Si)$, each vertex in $E(\Si)$ has valency $2$ while the valency of each vertex in $V(\Si)$ is equal to its valency in $\Si$.


The graph $\S(\Si)$ is closely related to the line graph $L(\Si)$ of $\Si$, the link arising via the distance $2$ graph $\S(\Si)^{[2]}$ of $\S(\Si)$.  As mentioned in the introduction, for connected graphs $\Si$, we can reconstruct $\Si$ from its subdivision graph $\ssig$. 
Indeed, for a connected graph $\Si$ (not $K_1$), $\ssig^{[2]}$ has two connected components, namely $\Si$ and $L(\Si)$. Moreover, either $\Si\cong L(\Si)\cong C_n$ for some $n\geqslant 3$, or $\Si$ is the unique connected component of $\ssig^{[2]}$ containing vertices of valency different from $2$ in $\ssig$.




The diameter of $\Si$ and $\S(\Si)$ are linked in the following way.

\begin{rem}\label{prop-main1}
Suppose that $\Sigma$ is a connected graph with $|V(\Sigma)|\geq 2$. Then
\begin{enumerate}
\item[(a)] Distances in $\Sigma$ and $\Gamma=\ssig$ are related as follows: for $\alpha, \beta \in V(\Gamma)$,
\[
d_{\Gamma}(\alpha, \beta)\left\{
  \begin{array}{ll}
    2d_{\Si}(\alpha, \beta) & \hbox{ if }\alpha \hbox{ and }   \beta \in V(\Si) ;\\
   2\min\{d_{\Si}(\alpha, u),d_{\Si}(\alpha, v)\}+1  & \hbox{ if } \alpha\in V(\Si) \hbox{ and } \\
   & \beta=\{u, v\} \in E(\Si); \\
   2\min\{d_{\Si}(x, y),d_{\Si}(x, v),& \hbox{ if } \alpha=\{x,u\} \hbox{ and }\\
   \quad \quad \quad d_{\Si}(y, u),d_{\Si}(u, v) \}+2&  \beta=\{y,v\}\in E(\Si);
  \end{array}
\right.
\]
\item[(b)] It follows easily that $\diam(\Gamma)=2\diam(\Sigma)+\delta$ for some $\delta\in\{0,1,2\}$.
\item[(c)] The inequality $0\leqslant \delta\leqslant 2$ cannot be improved (as is illustrated by the complete graphs on up to $4$ vertices). Note that $\delta=2$ if and only if $\Sigma$ contains two edges $e=\{x, y\}, f=\{u, v\}\in E(\Si)$ satisfying $d_{\Si}(x, u)=d_{\Si}(x, v)=d_{\Si}(y, u)=d_{\Si}(y, v)=\diam(\Si)$.
\end{enumerate}
\end{rem}

\section{Local distance and arc transitivity of $\ssig$}\label{sec:proofs1}

By Remark~\ref{prop-main1}, from now on we assume the following hypothesis.

\begin{hyp}\label{hyp}
$\Si$ is a connected graph of diameter $d$ with $|V(\Si)|\geqslant2$, such that $\ssig$ has diameter $2d+\delta$, for some $\delta\in \{0,1,2\}$, and  $G\leqslant \Aut(\Sigma)$. 
\end{hyp}

We study relationships between various symmetry properties of $\Sigma$ and $\ssig$. In particular we prove Theorems~\ref{thm:bary} and \ref{main3}.
First we consider the effect of local transitivity conditions on $\ssig$ related to edges of $\Si$. Note that the assumption $\Si\ne K_2$, under Hypothesis~\ref{hyp}, is equivalent to the condition that $\Si$ contains at least one 2-arc. 

\begin{lem}\label{lem3}
Suppose that Hypothesis~{\rm\ref{hyp}} holds, and in part (b) suppose that $\Si\ne K_2$. Set $\Ga:=\ssig$.
\begin{enumerate}
\item[(a)] If $G_e$ is transitive on $\Gamma_1(e)$ for all $e\in E(\Sigma)$, then $\Sigma$ is $G$-vertex transitive.
\item[(b)] $G_e$ is transitive on $\Gamma_1(e)$ and $\Ga_2(e)$ if and only if either
$\Sigma$ is $(G,2)$-arc transitive, or $\Si=C_n$ with $n$ even and $G\cong D_n$ has two orbits in $E(\Si)$.
\end{enumerate}
\end{lem}
\begin{proof}
(a) Let $u, v\in V(\Sigma)$.  There is a path $x_1, x_2, x_3, \cdots, x_n$ in $\Si$ such that $x_1=u, x_n=v$, since $\Sigma$ is connected. By assumption, for each $i$ there exists $g_i\in G_{\{x_i, x_{i+1}\}}$ such that $x_i^{g_i}=x_{i+1}$. The element $g_1g_2\cdots g_{n-1}$ maps $x_1=u$ to $x_n=v$.

(b) Let $e=\{u,v\}\in E(\Si)$. Then $\Gamma_1(e)=\{u, \, v\}$ and $\Gamma_2(e)(\Ga_1(u)\cup\Ga_1(v))\setminus\{e\}$. Moreover, $G_e$ is transitive on $\Ga_2(e)$ if and only if $G_e$ is transitive on $\Ga_1(e)$ and the stabiliser $G_{u,v}$ is transitive on $\Ga_1(u)\setminus\{e\}$ (or equivalently $G_{u,v}$ is transitive on
$\Si_1(u)\setminus\{v\}$).

\smallskip\noindent
$(\Longrightarrow)$ Suppose that, for all $e\in E(\Si)$,  $G_e$ is transitive on $\Gamma_i(e)$, for $i=1, 2$, and let $e=\{u,v\}\in E(\Si)$.
By Part (a), $G$ is transitive on $V(\Si)$, and so, since $\Si$ is connected and $\Si\ne K_2$, all vertices of $\Si$ have valency at least $2$.
As discussed above,
$G_{u,v}$ is transitive on $\Si_1(u)\setminus\{v\}$, and this holds for all edges $\{u,v'\}$ containing $u$. If $\Si$ has valency at least 3 then the subgroup $\langle G_{u,v'}\,|\,v'\in\Si_1(u)\rangle$ of $G_u$ is transitive on $\Si_1(u)$, and it follows that $G_u$ is 2-transitive on $\Si_1(u)$. In this case $\Si$ is $G$-vertex transitive by part (a), and locally $(G,2)$-arc transitive by Lemma~\ref{lem-smalls}(b), and hence $\Si$ is  $(G,2)$-arc transitive. On the other hand if $\Si$ has valency 2 then $\Si=C_n$ for some $n$. By part (a), $G$ is transitive on $V(\Si)$ and we have that $G_e$ interchanges the endpoints of $e$. Thus either $G=D_{2n}$ is 2-arc transitive on $\Si$, or $n$ is even and $G\cong D_n$ with two orbits on $E(\Si)$.

\smallskip\noindent
$(\Longleftarrow)$ In the exceptional case where $G=D_n$ acting with two edge orbits on $C_n$, the required properties of $G_e$ hold for all edges $e$. Thus we may suppose that $\Sigma$ is $(G,2)$-arc transitive. In particular  $G$ is transitive on $V(\Si)$, and so, since $\Si$ is connected and $\Si\ne K_2$, all vertices of $\Si$ have valency at least $2$.
Also $\Si$ is $(G,1)$-arc transitive so  $G_e$ is transitive on $\Gamma_1(e)$ for each edge $e$. Further, by Lemma~\ref{lem-smalls}(b)), $G_u$ is $2$-transitive on $\Si(u)$ for all $u\in V(\Si)$. Therefore, for $e=\{u,v\}\in E(\Sigma)$, $G_{u,v}$ is transitive on $\Si_1(u)\setminus\{v\}$. Thus by our remarks above,  $G_e$ is transitive on $\Ga_2(e)$ for all $e=\{u,v\}\in E(\Sigma)$,
\end{proof}


\subsection{Proof of Theorem~{\rm\ref{thm:bary}}.}
Let $\Ga=\ssig$, let $t:=\lceil\frac{s+1}{2}\rceil$, and note that $2t-1=s$ if $s$ is odd, and $2t-2=s$ if $s$ is even.

Suppose that $\Sigma$ is $(G, t)$-arc transitive. We prove first that $\Ga$ is locally $(G,2t-1)$-arc transitive.
Let $\alpha=(u_0, e_0,$ $u_1, e_1, ..., u_{t-1}, e_{t-1})$ and $\alpha'=(u_0,
e'_0, u'_1, e'_1, ..., u'_{t-1}, e'_{t-1})$ be $(2t-1)$-arcs in $\Gamma$ with
initial vertex $u_0\in V(\Si)$. Thus $\hat\alpha:=(u_0, u_1,  ... , u_{t-1}, u_t)$
and $\hat{\alpha'}:=(u_0, u'_1, \cdots, u'_{t-1}, u'_t)$ are $t$-arcs in $\Sigma$
where $e_{t-1}=\{u_{t-1}, u_t\}$, $e'_{t-1}=\{u'_{t-1}, u'_t\}$. By assumption
there exists  $g\in G_{u_0}$ such that $\hat\alpha^g = \hat{\alpha'}$,  and hence $\alpha^g=\alpha'$. Thus  $G_{u_0}$ acts transitively on $(2t-1)$-arcs in $\Gamma$ starting with $u_0$. Now consider $(2t-1)$-arcs in $\Gamma$ of the form $\beta=(e_0, u_1, e_1, \ldots, u_{s-1}, e_{t-1}, u_t)$ and $\beta'=( e_0, u'_1, e'_1, \ldots, u'_{s-1}, e'_{t-1}, u'_t)$ with initial vertex $e_0\in E(\Si)$. Now $e_0=\{u_0,u_1\}=\{u_0',u_1'\}$, and $\beta, \beta'$ correspond to $t$-arcs $\hat\beta:=(u_0, u_1, \ldots, u_{t-1}, u_t)$ and $\hat{\beta'}:=(u_0', u'_1, \ldots, u'_{t-1}, u'_t)$ in $\Sigma$. Hence,  there exists $g\in G$ such that $\hat\beta^g =\hat{\beta'}$. The element $g$ fixes $e_0$ setwise and therefore satisfies $\beta^g=\beta'$.  Thus $\Gamma$ is locally $(G, 2t-1)$-arc  transitive. Since $2t-1\geqslant s$, we have that   $\Gamma$ is locally $(G, s)$-arc  transitive.

Conversely suppose  $\Ga$ is locally $(G,s)$-arc transitive. We prove that $\Si$ is $(G,t)$-arc transitive. Consider  two $t$-arcs $\alpha=(u_0, u_1, \ldots, u_t)$ and $\alpha'=(u_0, u'_1, \dots, u'_t)$  in $\Si$ with initial vertex $u_0$. The corresponding $(2t-1)$-arcs in $\Ga$ are $\hat\alpha:=(u_0, e_0, u_1, e_1, \ldots, u_{t-1}, e_{t-1})$ and $\hat\alpha':=(u_0, e'_0, u'_1, e'_1, \cdots, u'_{t-1}, e'_{t-1})$, where for $i<t$, $e_i=\{u_i,u_{i+1}\}$ and $e_i'=\{u_i',u_{i+1}'\}$, and $u_0'=u_0$. Suppose first that $s$ is odd (so that $s=2t-1$). Then  $\Gamma$ is locally $(G, 2t-1)$-arc  transitive so there exists $g\in G_{u_0}$ such that $\hat\alpha^g=\hat\alpha'$. Thus $e^g_{t-1}=e'_{t-1}$ and $u^g_{t-1}=u'_{t-1}$, and hence $u^g_t=u'_t$ and $\alpha^g=\alpha'$. Therefore $\Si$ is locally $(G, t)$-arc transitive. By Lemma \ref{lem3}, $G$ is vertex transitive on $\Si$, and hence $\Si$ is $(G, t)$-arc transitive.

Finally suppose that $s$ is even, so that $s=2t-2$ and $\Ga$ is locally $(G,2t-2)$-arc transitive. Note that $t\geqslant 2$ in this case. Then the $(2t-1)$-arcs of $\Gamma$ in the previous paragraph are of the form $\hat\alpha=(\beta,e_{t-1})$ and $\hat\alpha'=(\beta',e_{t-1}')$ with $\beta,\beta'$ both $(2t-2)$-arcs in $\Gamma$ with initial vertex $u_0$. Since $\Ga$ is locally $(G,2t-2)$-arc transitive, there exists $g\in G_{u_{0}}$ such that
$\beta^g=\beta'$. Thus $\alpha^g=(\beta',e_{t-1}^g)$ and $e_{t-1}^g
=\{u_{t-1}^g,u_t^g\}=\{u_{t-1}',u_t^g\}$. Now we also have two $(2t-2)$-arcs
$\gamma:=(e_{0}', u_{1}',e_{1}',\ldots,u_{t-1}',e_{t-1}^g)$  and
$\gamma':=(e_{0}', u_{1}',e_{1}',\ldots,u_{t-1}',e_{t-1}')$ in $\Ga$ with initial vertex $e'_0$, so there exists $h\in G_{e'_0}$ such that $\gamma^h=\gamma'$. Since $\{u_0,u'_1\}=e'_0={e'_0}^h=\{u_0^h,{u_1'}^h\}=\{u_0^h,u_1'\}$, we have $u_0^h=u_0$ and hence $\alpha^{gh}=\alpha'$, with $gh\in G_{u_0}$.  Therefore $\Si$ is locally $(G, t)$-arc transitive, and as in the previous paragraph, $\Si$ is $(G, t)$-arc transitive.  \hfill $\square$ 

\subsection{Some consequences of Theorem~\ref{thm:bary}.}

We show for the cases $s=2, 3$, how to link local $(G,s)$-distance transitivity of $\ssig$ with symmetry properties of $\Si$.

\begin{cor}\label{basic}
Suppose that Hypothesis~{\rm\ref{hyp}} holds, and $\Si\ne K_2$. Then the following four conditions are equivalent.
\begin{enumerate}
\item[(a)]  $\ssig$ is locally $(G, 2)$-distance transitive.
\item[(b)] $\Sigma$ is $(G, 2)$-arc transitive.
\item[(c)] $\ssig$ is locally $(G, 3)$-arc  transitive.
\item[(d)] $\ssig$ is locally $(G, 3)$-distance transitive.
\end{enumerate}
\end{cor}
Since  $\Si\ne K_2$, $\Si$ contains a $2$-arc , so (b) is well-defined, and $\diam(\ssig)\geqslant 3$ by Remark~\ref{prop-main1}, so (a), (c) and (d) are well-defined.
\begin{proof}
By Theorem~\ref{thm:bary} for $s=3$, conditions (b) and (c) are equivalent.
It follows easily from the definition of local $(G,3)$-distance transitivity that condition (c) implies condition (d).
Also, by definition, condition (a) follows from condition (d). Thus it is sufficient to prove that condition (a) implies condition (b).

Let $\Gamma=\ssig$, and suppose that $\Gamma$ is locally $(G, 2)$-distance transitive. In particular then, for all $e\in E(\Si)$, $G_e$ is transitive on $\Gamma_1(e)$ and  $\Gamma_2(e)$, so by Lemma  \ref{lem3}(b), $\Sigma$ is $(G, 2)$-arc transitive,
or $\Si=C_n$ and $G=D_{n}$ with two edge orbits in $E(\Si)$. However in the latter case, condition (a) does not hold since $G_u=1$ for $u\in V(\Si)$.
\end{proof}

\begin{rem}\label{lem10}
Typically, a graph $\Si$ for which one of the equivalent conditions of Corollary~\ref{basic} holds will have girth at least $4$. Otherwise $\girth(\Si)=3$, and since $G$ is transitive on the $2$-arcs of $\Si$, all $2$-arcs form a $3$-cycle, and it follows that $\Si = K_n$ for some $n$. As $\Si\ne K_2$, $n\geqslant3$ and $G$ is a $3$-transitive subgroup of $S_n$.
\end{rem}

We finish this subsection with a result about local $(G, 4)$-distance transitivity when the girth of $\Si$ is greater than $4$.

\begin{prop}\label{3-arc}
Suppose that Hypothesis~{\rm\ref{hyp}} holds, and $\girth(\Si)\geqslant 5$. Then $\diam(\ssig)\geqslant 5$, and if $\ssig$ is locally $(G, 4)$-distance transitive, then $\Si$ is $(G,3)$-arc transitive.
\end{prop}
\begin{proof}
Let $\Ga=\ssig$.
Since $g=\girth(\Si)\geqslant 5$, $\Si$ contains a minimal cycle $(v_0,v_1,\ldots, v_{g-1})$, and so $d_\Ga(v_0,\{v_2,v_3\})=5$, thus $\diam(\Ga)\geqslant 5$.
Let $(u_0, u_1, u_2, u_3)$ and $(u_0', u'_1, u'_2, u'_3)$ be two $3$-arcs in $\Si$. Since $\Ga$ is locally $(G, 4)$-distance transitive
it follows from Corollary~\ref{basic} that
$\Si$ is $(G, 2)$-arc transitive. Thus there exists $a\in G$ such that $(u_0, u_1, u_2)^a=(u_0', u'_1, u'_2)$. Now we have two $4$-arcs  $\beta=(e'_0, u'_1, e'_1, u'_2, f)$ and $\beta'=(e'_0, u'_1, e'_1, u'_2, e'_2)$ in $\ssig$, where  $e'_i=\{u'_i, u'_{i+1}\}$ for $0\leqslant i\leqslant 2$ and $f=\{u'_2, u_3^a\}$. Thus $d_\Ga(e_0',f)\leqslant 4$. Also $e_0'\ne f$ since $u_2'\in f$. If $d_\Ga(e_0',f)=2$ then $e_0\cap f\ne\emptyset$, which is impossible since $g\geqslant5$. Hence $d_\Ga(e_0',f)=4$ and similarly $d_\Ga(e_0',e_2')=4$.
By the local $(G, 4)$-distance transitivity of $\Ga$, there exists $b\in G_{e'_0}$ such that $f^{b}=e_2'$.
If $b\in G_{e'_0}\setminus G_{u_0'}$, then
${u_0'}^{b}=u'_1,\ {u'_1}^{b}=u_0'$. Also ${u'_2}^{b}={u'}_2$ or ${u'_2}^{b}={u'}_3$, and
so the edge $\{u'_1,u'_2\}$ is mapped by $b$ onto $\{u'_0,u'_2\}$ or $\{u'_0,u'_3\}$ respectively, however those cannot be edges, as $g\geqslant 5$.
Thus $b\in G_{u_0',u_1'}$ and $f^{b}=e_2'$. If ${u_2'}^{b}=u_3'$, then
$\{u'_1,u'_2\}^{b}=\{u'_1,u'_3\}$ would be an edge,
again a contradiction. Hence ${u_2'}^{b}=u_2',\ {(u_3^{a})}^{b}=u_3'$, and we have that  $(u_0, u_1, u_2, u_3)^{ab}=(u_0', u'_1, u'_2, u'_3)$. Thus $\Si$ is $(G,3)$-arc transitive.
\end{proof}

\subsection{Proof of Theorem~\ref{main3}.}

The following Lemma \ref{lem-girth} is a critical ingredient in the proof of Theorem~\ref{main3}.

\begin{lem}\label{lem-girth}
Suppose that Hypothesis~{\rm\ref{hyp}} holds, and that $s$ is an even positive integer satisfying $s\leqslant 2d-1$. If  $\S(\Si)$ is  locally $(G, s)$-distance transitive, then $\girth(\Si)\geqslant s+2$.
\end{lem}
\begin{proof}
Let $\Ga=\ssig$, $s=2t$, and suppose that $\Ga$ is  locally $(G, 2t)$-distance transitive. By assumption $d\geqslant t+1\geqslant2$. Note that, by Lemma~\ref{lem3}, $G$ is transitive on $V(\Si)$. We prove the lemma by induction on $t$. If $t=1$, then $\Sigma$ is $(G,2)$-arc transitive by Corollary~\ref{basic}, and since $d\geqslant 2$, some $2$-arc does not lie in a $3$-cycle. Hence no $2$-arcs lie in a $3$-cycle and  $\girth(\Si)\geqslant4$.

Suppose now that $t\geqslant 2$ and the result holds for $t-1$. Then the conditions of the lemma hold for $t-1$, and so, by induction,   $\girth(\Si)\geqslant 2t$. We must show that $g:=\girth(\Si)\neq 2t, 2t+1$. Assume to the contrary that $g=2t$ or $2t+1$. Then $\Si$ contains a cycle $c=(u_0, u'_1, \cdots, u'_{g-1})$ of length $g$, and a vertex $v\in V(\Si)$ such that $d_{\Si}(u_0, v)=t+1$. There is a path  $p=(u_0, u_1, \cdots, u_t, u_{t+1})$
in $\Si$ with $u_{t+1}=v$. Since $\Ga$ is locally $(G, 2)$-distance transitive, there exists $a\in G_{u_0}$, such that $\{u_0, u_1\}^a=\{u_0, u'_1\}=e$, say. So $d_{\Si}(u_0, v^a)=d_{\Si}(u_0, v)=t+1$. Setting $f=\{u^a_t, v^a\}$ and $f'=\{u'_t, u'_{t+1}\}$, we have $d_{\Ga}(e, f)=d_{\Ga}(e, f')=2t$ since  $p$ is a shortest path from $u_0$ to $v$ and $c$ is a shortest cycle in $\Si$. Then by the  local $(G, 2t)$-distance transitivity of $\Ga$, there exists $b\in G_e$ such that $f^b=f'$. Since $b$ fixes $e$ setwise, ${u_0}^b\in \{u_0, u'_1\}$. Also $t+1=d_{\Si}(u_0, v^a)=d_{\Si}({u_0}^b, v^{ab})$ and $v^{ab}\in f'$.    Therefore  $u_0^b$ and $v^{ab}$ are vertices in $c$ at distance $t+1$ in $\Si$, which is a contradiction. Hence $g\geqslant 2t+2$. \end{proof}

\noindent
\emph{Proof of Theorem~\ref{main3}.}\quad If $s=1$ then the claimed equivalence follows from Lemma~\ref{lem-smalls} and Theorem~\ref{thm:bary} (with $s=1$). Also, if $s=2$ or 3, then the equivalence follows from Corollary~\ref{basic}. Thus we may assume that  $s\geqslant4$. Let $s'$ be the largest even integer $s'\leqslant s$, and let $t:=\lceil\frac{s+1}{2}\rceil =\frac{s'}{2}+1$. Note that $t\geqslant3$.

Let $\Ga=\ssig$ and suppose that $\Ga$ is locally $(G,s)$-distance transitive.
Then by Lemma~\ref{lem-girth},  $\girth(\Si)\geqslant s'+2\geqslant s+1$. Consider
two $t$-arcs $\alpha=(u_0, u_1, \cdots, u_t)$ and $\alpha'=(u_0, u'_1, \cdots, u'_t)$ of $\Si$. These correspond to two $(s'+2)$-arcs $\beta=(u_0, e_0, u_1, e_1 \cdots,\ u_{t-1}, e_{t-1},u_t)$ and $\beta'=(u_0, e'_0, u'_1, e'_1 \cdots, u'_{t-1},$ $e'_{t-1},u'_t)$ of  $\Ga$, where $u_0=u'_0$, $e_i=\{u_i, u_{i+1}\}$ and $e'_i=\{u'_i, u'_{i+1}\}$ for each $i<t$. Now $\Ga$ is  locally $(G, 2)$-distance transitive and $u_1,u'_1\in\Ga_2(u_0)$, so there exists $a\in G_{u_0}$ such that $u_1^a=u'_1$. Therefore $e_0^a=e'_0$ and $\beta^a=(u_0, e'_0, u'_1, e^a_1 \cdots, {u}_{t-1}^{a}, {e}_{t-1}^{a},u_t^a)$.
For each $r<t$ the initial $(2r+1)$-arc $(u_0, e_0,$ $u_1, e_1 \cdots, u_r, e_r)$ of $\beta$ is a path in $\Ga$ of length $2r+1$ and hence  $d_{\Ga}(u_0, e_r)\leqslant2r+1$. If there were a shorter path in $\Ga$ from $u_0$ to $e_r$, then we would obtain a cycle in $\Ga$ of length at most $(2r+1)+(2r-1)=4r$, corresponding to a cycle in $\Si$ of length at most $2r\leqslant 2t-2 = s' < \girth(\Si)$, which is a contradiction.
Hence $d_{\Ga}(u_0, e_r)=2r+1$ for each $r<t$. Thus $d_{\Ga}(u_0, {e_{t-1}}^a)=d_{\Ga}(u_0, e_{t-1})=2t-1$, and it follows that
$d_{\Ga}(e'_0, {e_{t-1}}^a)=d_{\Ga}(e_0, e_{t-1})=2t-2$ and also
$d_{\Ga}(e'_0, e'_{t-1})=2t-2$ for similar reasons. Since $2t-2=s'\leqslant s$, the graph $\Ga$ is
locally $(G, s')$-distance transitive, so there exists $b\in G_{ e'_0}$  such that  ${e_{t-1}}^{ab}=e'_{t-1}$. The element $b$ fixes $u_0$ and $u'_1$, since otherwise we would have $2t-1=d_{\Ga}(u_0, {e_{t-1}}^a)=d_{\Ga}({u_0}^b, {e_{t-1}}^{ab})=d_{\Ga}(u'_1, e'_{t-1})=2t-3$, which is a contradiction. Hence $b\in G_{u_{0}}$, and so  $ab\in G_{u_{0}}$. Similarly, since  $d_{\Ga}(u_0, e_r)=2r+1$, for all $r< t$, we see that $b$ maps each ${u_r}^a$ to $u'_r$ and ${e_r}^a$ to $e'_r$. Thus $b$ also maps  ${u_t}^a$ to $u'_t$, and so $\alpha^{ab}=\alpha'$. Thus $\Si$ is locally $(G, t)$-arc transitive. Since, by Lemma \ref{lem3},
 $\Si$ is $G$-vertex transitive, $\Si$ is $(G, t)$-arc transitive.

Conversely suppose that $\Si$ is $(G,t)$-arc transitive. Then by Theorem~\ref{thm:bary}, $\Ga$ is locally $(G,s)$-arc transitive. Let $r\leqslant s$ and let $x,y,y'\in V\Ga$ with $d_\Ga(x,y)=d_\Ga(x,y')=r$. Then there are $r$-arcs $\beta$ and $\beta'$ in $\Ga$ from $x$ to $y$ and $y'$ respectively. Since $r\leqslant s$, $G_x$ is transitive on the $r$-arcs with initial vertex $x$, and so there exists $g\in G_x$ such that $\beta^g=\beta'$, and hence $y^g=y'$. Thus $\Ga$ is locally $(G,s)$-distance transitive. \hfill $\square$

\section{Local distance transitivity of $\ssig$ for $\Sigma$ in Table~\ref{table}}\label{sec-examples}


Suppose that Hypothesis~{\rm\ref{hyp}} holds, for $G, \Si, d, \delta$, and that $s\leqslant \diam(\ssig)$. 

\subsection{$\Sigma=K_{n}$ with $n\geqslant 2$}\label{ex-K23}

We easily see that $\diam(\Sigma)=1$, $\Aut(\Sigma)=S_{n}$ and $\diam(\ssig)=\min\{n,4\}$.

For $n=2,3$, the graph $\ssig$ is locally $(G,s)$-distance transitive if and only if $G=S_n$ if and only if $\ssig$ is locally $G$-distance transitive.
\begin{prop}\label{ex-Kn}
Let $\Sigma=K_n$ with $n\geqslant 4$. Then the graph $\ssig$ is locally $(G,s)$-distance transitive if and only if $G$ and $s$ are as in Table~\ref{tbl:kn}. 
\end{prop}
\begin{table}[h]
\begin{center}
\caption{$(G,s)$-distance transitivity for $\Si=K_n$ and $n\geq 4$.}\label{tbl:kn}\smallskip
\begin{tabular}{l|l}
\hline
$s$& Conditions on $G$\\ \hline
$1$ & $2$-transitive on $V(\Si)$\\
$2,3$ & $3$-transitive on $V(\Si)$\\
$4$ & $4$-transitive on $V(\Si)$, or $n=9$ and $G=\PGaL(2,8)$\\ \hline
\end{tabular}
\end{center}
\end{table}
\begin{proof}
By Theorem~\ref{main3} and Corollary~\ref{basic}, for $s=1,2$ respectively, $\ssig$ is locally $(G,s)$-distance transitive if and only if $\Si$ is $(G,s)$-arc transitive, or equivalently, $G$ is $(s+1)$-transitive on $V(\Si)$. Also $\ssig$ is locally $(G,3)$-distance transitive if and only if $\Si$ is $(G,2)$-arc transitive, that is, $G$ is 3-transitive.
Finally if $\Ga=\ssig$ is locally $(G,4)$-distance transitive, then for $e=\{1,2\}$, $G_e$ is transitive on $\Ga_4(e)=\{\{i,j\}\,|\, i>j>2\}$, and it follows that $G$ is transitive on $4$-subsets of $V(\Si)$. By \cite[Theorem 9.4B]{Dixon96}, either $G$ is 4-transitive (and all such groups act locally 4-distance transitively on $\Ga$),
or $n=9$, $G=\PGL(2,8)$ or $\PGaL(2,8)$, or $n=33$, $G=\PGaL(2,32)$.
The groups $\PGL(2,8)$ and $\PGaL(2,32)$ do not arise since in these cases $|\Ga_4(e)|$ does not divide $|G_e|$. On the other hand if $n=9$ and $G=\PGaL(2,8)$, then $\Ga$ is locally $(G,3)$-distance transitive by Corollary~\ref{basic}, $\Ga_4(v)=\emptyset$ for $v\in V(\Si)$ and   $G_e$ is transitive on both $\Ga_2(e)$ and $\Ga_4(e)$ by \cite[p.6]{Atlas}; it follows that $\Ga$ is locally $(G,4)$-distance transitive. 
\end{proof}

\subsection{$\Sigma=K_{n,n}$ with $n\geqslant 2$}

We have $\diam(\Sigma)=2$, $\diam(\ssig)=4$ and $\Aut(\Sigma)=S_{n}\wr S_{2}$ acting imprimitively on the vertices with the two biparts forming an invariant vertex partition. If a subgroup $G$ is vertex transitive then, by the Embedding Theorem~\cite[Theorem 8.5]{BMMN} for imprimitive groups, we may assume that $G\leqslant H\wr S_2$ with $G\cap(H\times H)$ projecting onto $H$ in each component. In this case the group $H\leqslant S_n$ is called the component of $G$.

\begin{prop}\label{ex-Knn}
Let $\Delta_1$ and $\Delta_2$ be biparts of $\Sigma=K_{n,n}$ with $n\geqslant 2$. Then the graph $\ssig$ is locally $(G,4)$-distance transitive if and only if (i) $G$ is transitive on $V(\Si)$ with $G\leqslant H\wr S_2$ having component $H$, (ii) $H$ is 2-transitive on each $\Delta_i$ and, for $u_1\in\Delta_1$, $G_{u_1}$ is transitive on $(\Delta_1\setminus\{u_1\})\times \Delta_2$, and (iii) for $u_1\in\Delta_1$ and $u_2\in\Delta_2$, the stabiliser $G_{\{u_1,u_2\}}$ interchanges $u_1$ and $u_2$, and is transitive  on $\{\{v_1,v_2\} \mid v_i\in \Delta_i\setminus\{u_i\}\,\}$.
 In particular, $\ssig$ is locally $G$-distance transitive if and only if $G$ satisfies these three conditions.
\end{prop}
\begin{proof}
Let $\Ga=\ssig$. It is not difficult to show that conditions (i)--(iii) together ensure that $\Ga$ is locally $(G,4)$-distance transitive. So suppose conversely that $G$ is such that $\Ga$ is locally $(G,4)$-distance transitive. By Lemma~\ref{lem3}(a), $G$ is transitive on $V(\Si)$, so we may assume that $G\leqslant H\wr S_2$ with component $H$ and $H$ acts transitively on each $\Delta_i$. Let   $u_1\in\Delta_1$ and $u_2\in\Delta_2$, and $e=\{u_1,u_2\}$. Since, $G_{u_1}$ is transitive on $\Ga_2(u_1)=\Delta_2$ and $\Ga_4(u_1)=\Delta_1\setminus\{u_1\}$, and these sets have coprime sizes, it follows that $G_{u_1}$ is transitive on $(\Delta_1\setminus\{u_1\})\times \Delta_2$. In particular $H$ is 2-transitive on each $\Delta_i$. Also, since $G_e$ is transitive on $\Ga_1(e)$ and $\Ga_4(e)$, it follows that $G_{\{u_1,u_2\}}$ interchanges $u_1$ and $u_2$, and is transitive on  $\{\{v_1,v_2\}\,|\, v_i\in \Delta_i\setminus\{u_i\}\,\}$.  The final assertion holds since $\diam(\Ga)=4$.
\end{proof}



\subsection{$\Sigma=C_{n}$ for $n\geqslant 3$}
Here we consider  $C_n$ in general instead of $C_5$ (which is in Table~\ref{table}) because we will need it in the proof of Theorem \ref{thm-char}(a).
We have $\diam(\Si)=\lfloor\frac{n}{2}\rfloor$, $\diam(\ssig)=n=\girth(\Si)$ (so $\delta=0$ if $n$ is even and $\delta=1$ if $n$ is odd), and   $\Aut \Si=D_{2n}$.

 Also $\diam(\Si)=2$, $\diam(\ssig)=5$ and   $\Aut(\Sigma)=D_{10}$, and:

\begin{prop}\label{ex-Cn}
The graph $\ssig$ is locally $(G,s)$-distance transitive for some $s\leqslant n$ if and only if $G=D_{2n}$  and  if and only if  $\ssig$ is locally $G$-distance transitive.
\end{prop}


\subsection{$\Sigma=\rmP$, the Petersen graph}

Here $\diam(\Si)=2$, $\diam(\ssig)=6$  and   $\Aut(\Sigma)=S_5$.
\begin{prop}\label{ex-O3}
The graph $\ssig$ is locally $(G,s)$-distance transitive, for $s=4$ or for $s=5$, if and only  $G=S_5$. Moreover, $\ssig$ is locally $S_5$-distance transitive.
\end{prop}
\begin{proof}
Let $\Ga=\ssig$ and suppose that $\Ga$ is locally $(G,4)$-distance transitive. Then
by Lemma~\ref{lem3}(a), $G$ is transitive on $V(\Si)$. Since $|V(\Sigma)|=10$, $|\Gamma_{4}(v)|=4$ and $|\Gamma_{4}(e)|=8$, for $v\in V(\Sigma)$ and $e\in E(\Sigma)$, we have that $120\mid |G|$, and so $G=S_{5}$. Conversely, let $G=S_{5}$. Then we have $G_{v}=S_{3}\times S_{2}$, $G_{e}=D_{8}$ and we can easily check that $\ssig$ is locally $G$-distance transitive.
\end{proof}

\subsection{$\Sigma=\HoSi$, the Hoffman-Singleton graph}
The Hoffman-Singleton graph is a regular graph  of valency $k=7$ and diameter $2$ with the largest possible number $k^2+1$ of vertices, see \cite{HoSi}.
Here, $\diam(\Sigma)=2$, $\diam(\ssig)=6$  and   $\Aut(\Sigma)=H.2$, where $H=\PSU(3, 5)$.
\begin{prop}\label{ex-HoSi}
The graph $\ssig$ is locally $(G,s)$-distance transitive, for $s=4$ or for $s=5$, if and only  $G=H$ or $H.2$. Moreover, $\ssig$ is locally $H$-distance transitive and locally $H.2$-distance transitive.
\end{prop}
\begin{proof}
Let $\Ga=\ssig$ and suppose that $\Ga$ is $(G,4)$-distance transitive. Then
by Lemma~\ref{lem3}(a), $G$ is transitive on $V(\Si)$. Since
$|V(\Sigma)|=50$, $|\Gamma_{1}(v)|=7$ and $|\Gamma_{4}(e)|=72$, for $v\in V(\Sigma)$ and $e\in E(\Sigma)$, we have that $2^3.3^2.5^2.7=|H|/10$ divides $|G|$, and hence $G=H$ or $H.2$ by \cite[p. 34]{Atlas}. On the other hand, for $G=H$, it follows from \cite[p. 34]{Atlas} that $G_v=A_7$ and $G_e=M_{10}$, and so we can easily prove that $\Gamma$ is locally $X$-distance transitive for $X=H$ or $H.2$.
\end{proof}

\section{Proofs of Theorem~\ref{thm-char} and Corollary~\ref{ldt}.}\label{sec-proof}

\noindent \emph{ Proof of Theorem \ref{thm-char}(a). }\quad
Suppose that $2d\leqslant s\leqslant 2d+\delta$ and that $s\geqslant 15+\delta$. Then $d\geqslant 8$ and so $s\geqslant \max\{16, 15+\delta\}$.  Let $\Ga=\ssig$. Suppose first that $\Ga=\S(\Si)$ is locally $(G,s)$-distance transitive.
Set $s':=s-\delta-1$. Then $s'\leqslant 2d-1$ and $\Ga=\S(\Si)$ is locally $(G,s')$-distance transitive. By Theorem~\ref{main3}, $\Si$ is $(G,t)$-arc transitive, where $t=\lceil\frac{s'+1}{2}\rceil$. If $s=16$ then $16\geqslant 15+\delta$ so $\delta\leqslant 1$ and $s'=s-\delta-1\geqslant 14$. On the other hand if $s\geqslant17$, then $s'=s-\delta-1\geqslant 14$. Thus in both cases $s'\geqslant14$ and hence $t\geqslant 8$. By Weiss' Theorem  \cite{Weiss1981},
it follows that $\Si=C_n$ for some $n$. Since $\ssig$ is locally $(G,s)$-distance transitive, $G=D_{2n}$ by Proposition \ref{ex-Cn}. If $n$ is even, then $\delta=0$, and so $s=2d=n$, whence $\Si=C_s$. If $n$ is odd, then $\delta=1$, and so $s=2d=n-1$ or $s=2d+1=n$, giving $\Si=C_{s+1}$ or $C_s$, respectively. Notice that in the former case, $s$ is even.

Conversely suppose that $\Si=C_s$ and $G=D_{2s}$  or $s$ is even, $\Si=C_{s+1}$ and $G=D_{2s+2}$.  Then $\ssig$ is locally $G$-distance transitive by  Proposition \ref{ex-Cn}, and since $\diam(\ssig)\geqslant s$, we have that $\ssig$ is locally $(G,s)$-distance transitive. Notice that $s\geqslant 16$, as $\delta=1$ for $\Si=C_s$ with $s$ odd.
\hfill $\square$

\bigskip
\noindent\emph{Proof of Theorem~\ref{thm-char}(b).}\quad
Suppose that Hypothesis~\ref{hyp} holds and that $2d\leqslant s\leqslant 2d+\delta$ and $s\leqslant5$.
In particular $d\leqslant2$.  Suppose first that $d=1$. Then $\Si=K_n$ for some $n\geqslant2$. If $n=2$ or $n=3$ then the assertions of Theorem~\ref{thm-char} and the entries in Table~\ref{table} hold, see Section \ref{ex-K23}. If $n\geqslant4$, then these hold by Proposition~\ref{ex-Kn}.

Suppose now that $d=2$, so $4\leqslant s\leqslant 4+\delta$ and by assumption $s\leqslant 5$.
Let $\Ga=\ssig$. Suppose first that $\Ga$ is locally $(G,s)$-distance transitive. Then in particular $\Ga$ is   locally $(G,3)$-distance transitive, and so by Theorem~\ref{main3}, $\Si$ is $(G,2)$-arc transitive. Note that the girth of $\Si$ is at most 5 (since $d=2$)
and at least 4 (since $\Si$ is $(G,2)$-arc transitive and $d\ne1$, see Remark \ref{lem10}).

We now prove that if $\girth(\Si)=4$ then $\Si=K_{n,n}$, for some $n\geqslant2$.
Let $(u_0. u_1, u_2, u_3)$ be a $4$-cycle in $\Si$ and set $e_i=\{u_i, u_{i+1}\}$ for $0\leqslant i\leqslant 2$. Then $\alpha :=(e_0, u_1, e_1, u_2, e_2)$ is a $4$-arc in $\Ga$, and $d_{\Ga}(e_0, e_2)=4$ since $e_0\cap e_2=\emptyset$. Thus $\Ga_4(e_0)\neq \emptyset$. For every edge $e'_1=\{u_1, u'_2\}$ containing $u_1$, $d_{\Si}(u_0, u'_2)=2$ since $\girth(\Si)>3$. For the same reason, for every edge $e'_2=\{u'_2, u'_3\}$ containing $u'_2$ with $e'_2\neq e'_1$, we have $d_\Si(u_3',u_1)=2$ and $u_3'\ne u_0$, so $e_0\cap e_2'=\emptyset$ and $d_{\Ga}(e_0, e'_2)=4$. Thus $e_2, e'_2\in \Ga_4(e_0)$ and so $e^a_2=e'_2$ for some $a\in G_{e_0}$. This implies that $u_3^a\in\Si_1(u_0^a)\cap \Si_2(u_1^a)\cap e_2'$. If $a$ fixes $u_0$ and $u_1$, then we conclude that $u_3^a=u_3'\in\Si_2(u_1)$ and hence $u_3'\in\Si_1(u_0)$. On the other hand, if $a$ does not fix $u_0$ then it must interchange $u_0$ and $u_1$, and we must have $u_3^a=u'_2\in\Si_1(u_1)\cap e_2'$, which implies that $u_3'=u_2^a\in\Si_1(u_1^a)=\Si_1(u_0)$. In either case $u_3'$ is adjacent to $u_0$ in $\Si$.
Thus $\Si_1(u_2')\subseteq\Si_1(u_0)$ for each $u_2'\in\Si_2(u_0)$. It follows that $\Si_1(u_2')=\Si_1(u_0)$ since $u_0$ and $u'_2$ have the same valency, and hence that $\Si$ is a complete bipartite graph $K_{n,m}$ for some $n,m\geqslant2$. Since $G$ is transitive on $V(\Si)$, $n=m$.

It follows from Proposition~\ref{ex-Knn} that  the assertions of Theorem~\ref{thm-char} and the entries in Table~\ref{table} hold if $\Si=K_{n,n}$ with $n\geqslant2$, and hence if $\girth(\Si)=4$. Thus we may assume that $\girth(\Si)=5$. In this case, it follows from Proposition \ref{3-arc}, that $\Si$ is $(G,3)$-arc transitive. Therefore by \cite[Lemma 3.4]{Ivanov87}, $\Si$ is one of the graphs $C_5$, $\rmP$ or  $\HoSi$. It now follows, from Propositions \ref{ex-Cn}, \ref{ex-O3}, and \ref{ex-HoSi}, that
the assertions of Theorem~\ref{thm-char} and the entries in Table~\ref{table} hold in the case where $\girth(\Si)=5$. This completes the proof. \hfill$\square$

\bigskip
\noindent\emph{Proof of Corollary~\ref{ldt}.}\quad
Suppose that $\Si$ is a connected graph of diameter $d\leqslant 2$ such that $|V(\Si)|\geqslant2$ and  let $G\leqslant \Aut(\Sigma)$. Thus $\diam(\ssig)=2d+\delta\leqslant 6$.  Suppose that $\ssig$ is locally $G$-distance transitive.
If $2d+\delta\leqslant 5$ then, applying Theorem~\ref{thm-char}(b) with $s=2d+\delta$, we find that $\Si, G$ are as in Table~\ref{table}. Thus we may assume that $\diam(\ssig)=2d+\delta= 6$, so $(d,\delta)=(2,2)$.
Then Theorem~\ref{thm-char}(b) applies with $s=5$, again yielding the graphs in Table~\ref{table}.  The local $G$-distance transitivity of the graphs in Table~\ref{table} follows from the properties proved in Section~\ref{sec-examples} about these graphs.
\hfill $\square$

\section*{Acknowledgements}

The first author thanks Bu-Ali Sina University and the University of Western Australia for support during her sabbatical leave.  The authors thank anonymous referees for helpful suggestions with the exposition.  The third author is supported by Australian Research Council Federation Fellowship FF0776186. This paper forms part of the Federation Fellowship project.



\section*{References}






\end{document}